\documentclass[submission,copyright,creativecommons]{eptcs}
\providecommand{\event}{AiML 2026} 

\usepackage{aiml26}

\usepackage{iftex}

\ifpdf
  \usepackage{underscore}         
  \usepackage[T1]{fontenc}        
\else
  \usepackage{breakurl}           
\fi

\usepackage{hyperref}
\usepackage{amsthm,amsmath,amssymb,amsfonts,multicol}
\usepackage{tikz}

\hypersetup{
    colorlinks=true,
    linkcolor=black,
    urlcolor=blue,
    citecolor=black,
}

\theoremstyle{definition}

\newcommand{\tuple}[1]{{\langle #1 \rangle}}

\newcommand{\ps}{\Diamond}
\newcommand{\domai}[1]{W^{#1}}
\newcommand{\nec}{\Box}
\newcommand{\lanfull}{\mathcal L}
\newcommand{\define}[1]{\emph{#1}}
\newcommand{\bfrm}[1]{#1^\nec}
\newcommand{\dfrm}[1]{#1^\ps}

\newcommand{\from}{\colon}

\newcommand{\veryshortarrow}{\to}

\newcommand{\ignore}[1]{}
\newcommand{\Model}{{\mathcal M}}
\newcommand{\Frame}{{\mathcal F}}
\newcommand{\mt}[1]{{#1}_{\rm top}}
\newcommand{\ms}[1]{{#1}_{\rm ds}}
\newcommand{\md}[1]{{#1}^{d}}
\newcommand{\mc}[1]{{#1}^{c}}

\newcommand{\ax}[1]{{\rm #1}}
\newcommand{\val}[1]{\|#1\|}
\newcommand{\peq}{\preccurlyeq}
\newcommand{\seq}{\succcurlyeq}
\newcommand{\rel}{\sqsubset}
\newcommand{\ler}{\sqsupset}
\newcommand{\brel}{\ll}
\newcommand{\bler}{\gg}

\usepackage{yfonts,array}
\usepackage{amssymb,amsfonts,amsmath,amsthm,units,nicefrac,stmaryrd}
\usepackage{cleveref}
\usepackage{upgreek,xcolor,bm,tikz,stackrel}
\usepackage[inline]{enumitem}
\usepackage{graphicx}
\usepackage[all]{xy}
\usepackage{tikz-cd}
\usepackage{multicol}
\usepackage{color}
\usepackage[inline]{enumitem}
\usepackage{mathtools}
\usepackage{tasks}
\usetikzlibrary{arrows}

\newcommand{\KF}{\mathsf{K4}}
\newcommand{\ckf}{\mathsf{CK4}}
\newcommand{\ikf}{\mathsf{IK4}}
\newcommand{\kfi}{\mathsf{K4I}}
\newcommand{\gkf}{\mathsf{GK4}}
\newcommand{\gkfc}{\mathsf{GK4^c}}
\newcommand{\SF}{\mathsf{S4}}

\newcommand{\csf}{\mathsf{CS4}}
\newcommand{\isf}{\mathsf{IS4}}
\newcommand{\sfi}{\mathsf{S4I}}
\newcommand{\gsf}{\mathsf{GS4}}
\newcommand{\gsfc}{\mathsf{GS4^c}}

\newcommand{\taui}{\tau_{\veryshortarrow}}
\newcommand{\taud}{\tau_\Diamond}
\newcommand{\taub}{\tau_\Box}
\newcommand{\inti}{i_{\veryshortarrow}}
\newcommand{\intd}{e_\Diamond}
\newcommand{\intb}{e_\Box}

\newcommand{\clod}{d_\Diamond}

\title{Polytopological Semantics for Intuitionistic Modal Logics}
\author{Juan P. Aguilera
\institute{Institute of Discrete Mathematics and Geometry\\
Technische Universit\"at Wien\\
Vienna, Austria}
\email{aguilera@logic.at}
\and
David Fernández-Duque
\institute{Department of Philosophy\\
University of Barcelona\\
Barcelona, Spain}
\email{fernandez-duque@ub.edu}
\and
Leonardo Pacheco
\institute{School of Computing\\Institute of Science Tokyo\\Tokyo, Japan}
\email{leonardovpacheco@gmail.com}
}

\newcommand{\titlerunning}{Polytopological Semantics for Intuitionistic Modal Logics}
\newcommand{\authorrunning}{J. Aguilera, D. Fernández-Duque \& L. Pacheco}

\hypersetup{
  bookmarksnumbered,
  pdftitle    = {\titlerunning},
  pdfauthor   = {\authorrunning},
  pdfsubject  = {EPTCS},               
  pdfkeywords = {non-classical modal logic, strong completeness, topological semantics} 
}

\begin{document}
\maketitle

\begin{abstract}
    We develop polytopological semantics for various constructive, intuitionistic, and Gödel--Dummett variations of $\mathsf{K4}$ and $\mathsf{S4}$. In our models, intuitionistic and modal operators are interpreted via various topologies over a single set, equipped with either the closure or derivative operators. We identify regularity conditions to ensure that spaces validate each of our target logics and prove that all the logics considered are sound and strongly complete with respect to their respective semantics.
\end{abstract}

\section{Introduction}
\label{section::introduction}

Both intuitionistic logic and the modal logics $\KF$ and $\SF$ (along with many related logics) enjoy topological semantics dating back to McKinsey and Tarski~\cite{mctarski}, predating even Kripke semantics.
It is only natural to expect for intuitionistic modal logics to also enjoy a topological interpretation, and indeed there have been some efforts in this direction~\cite{Davoren07,Fernandez-Duque18,degroot2025iS4,TomaszWitczak2019}.
However, none of these proposals simultaneously treat intuitionistic implication and both modalities via independent topologies according to the McKinsey--Tarski semantics.
This would mirror the relational setting, where semantics are given by structures $\tuple{W,\peq,\rel }$, consisting of a set equipped with two preorders, possibly satisfying additional confluence properties.
A natural topological adaptation would be to equip a set with distinct topologies $\taui$ and $\taud$.
This raises the question of whether this leads to a reasonable interpretation of constructive or intuitionistic logic and what, if any, compatibility conditions on the topologies are required to obtain a well-behaved semantics.

As we will see, a triple of topologies $\taui$, $\taud$, and $\taub$ on a single set give rise for an interpretation of $\csf$, with the only caveat that $\taub$ should be a subset of the meet of $\taui$ and $\taud$, \emph{i.e.}, via $\tau_\Box \subseteq \taui \wedge \taud$.
The latter condition mimics the situation for relational semantics, where $\Box$ is interpreted by quantifying over both $\peq$ and $\rel$, as required to obtain semantics satisfying the usual upwards persistence of intuitionistic truth.
However, if we wish to provide semantics for stronger logics, such as the well-studied $\isf$, additional constraints are needed governing the interaction between the topologies.
Likewise, non-classical versions of $\KF$ can be obtained by interpreting the modalities as the Cantor derivative over spaces satisfying the $T_d$ separation axiom.

For each of the logics $\ckf,\ikf,\kfi,\gkf,\gkfc,\csf,\isf,\sfi,\gsf,\gsfc$, we define a collection of polytopological models.
We prove strong completeness results for each of these logics over their respective Alexandroff models based on old and new strong completeness results.
Recall that Alexandroff spaces are topological spaces whose open sets are the upsets of some topology, and for logics above $\SF$ allow us to interpret Kripke frames as a special case of topological spaces.

We note that many of the completeness results for birelational Kripke semantics we prove have appeared before.
The strong completeness for the intuitionistic logics $\ikf$ and $\isf$ already appear in \cite{servi1984axiomatizations}.
Similarly, strong completeness for variants of $\SF$ already appeared \cite{balbiani2024variants}.
A completeness results for $\ckf$ with respect to a sequent calculus appears in \cite{ArisakaDS15}, but no semantic result is described there.
As far as we are aware, the logics $\kfi$, $\gkf$, and $\gkfc$ have not been studied before.

\paragraph{Outline}
On Section~\ref{sec::syntax}, we briefly describe the syntax and logics studied in this paper.
On Section~\ref{sec::general-semantics}, we define a general polyderivative semantics which will cover both birelational and polytopological semantics.
On Section~\ref{sec::soundness}, we prove the soundness of each logic with respect to their polyderivative models.
On Section~\ref{sec::birelational-semantics}, we define birelational models for each logic.
On Section~\ref{sec::topological-completeness}, we define polytopological models for each logic and transfer the strong completeness results from birelational to polytopological semantics.
On Section~\ref{sec::conclusion}, we discuss some directions for future work.

\section{Syntax}
\label{sec::syntax}

We will work within a standard language of intuitionistic modal logic, which has the peculiarity that,  unlike classical modal logic, $\Box$ and $\Diamond$ are not inter-definable.
Fix a set $\mathrm{Prop}$ of proposition symbols.
The \emph{modal formulas} are defined by the following grammar:
\[
    \varphi := \bot \mid p\mid \varphi\land\varphi  \mid \varphi\lor\varphi \mid \varphi\to\varphi \mid \Box\varphi \mid \Diamond\varphi.
\]
As usual, we define $\neg\varphi:= \varphi\to\bot$ and $\top:=\bot\to\bot$.

\begin{definition}
    A \emph{(modal) logic} is a set of formulas closed under necessitation and \emph{modus ponens,} defined as follows:
    \begin{multicols}{2}
    \begin{description}
        \item[\ax{Nec}]  $:=\frac{\varphi}{\Box\varphi}$,
        \item[{\ax{MP}}] $:=\frac{\varphi \;\;\; \varphi\to\psi}{\psi}$.
    \end{description}
    \end{multicols}
    $\ckf$ is the least logic containing the all intuitionistic tautologies, and the axioms:
    \begin{multicols}{2}  
    \begin{description}
        \item[${\ax K}_\Box$] $:= \Box(\varphi\to\psi) \to (\Box\varphi \to \Box \psi)$,
        \item[${\ax 4}_\Box$] $ := \Box\varphi \to \Box\Box\varphi$,
        \item
        \item[${\ax K}_\Diamond$] $:= \Box(\varphi\to\psi) \to (\Diamond\varphi \to \Diamond \psi)$,
        \item[${\ax 4}_\Diamond$] $ := \Diamond\Diamond\varphi \to \Diamond \varphi$,
        \item[\ax{N}] $  := \neg \Diamond \bot$.
    \end{description}
     \end{multicols}
\end{definition}

\begin{remark}
    Note that \ax{N} is typically not included in $\ckf$, but we add it in order to obtain a more uniform treatment of the various logics we consider.
    Semantically, it corresponds to the condition that $\bot$ is false on all possible worlds, \emph{i.e.}, that models are {\em infallible}; such logics have been studied by Wijesekera~\cite{wijesekera1990constructive}.
    In the presence of transitive modalities, constructive logics are naturally embeddable into Wijesekera logics~\cite{SantiagoFJ}, so there is no loss of generality in working with the latter.
\end{remark}

All other logics we consider are extensions of $\ckf$, with combinations of the following axioms:
\begin{multicols}{2}
\begin{description}
    \item[\ax{$T_\Box$}] $:= \Box \varphi\to \varphi$,
    \item[\ax{FS}] $ := (\Diamond \varphi \to \Box\psi) \to \Box(\varphi\to\psi)$,
    \item[\ax{DP}] $ := \Diamond (\varphi\lor\psi) \to \Diamond\varphi\lor\Diamond\psi$,
    \item[\ax{$T_\Diamond$}] $:= \varphi\to\Diamond\varphi$,
    \item[\ax{RV}] $ := \Box(\varphi\lor\psi) \to \Box\varphi\lor\Diamond\psi$,
    \item[\ax{GD}] $ := (\varphi\to\psi) \lor (\psi\to\varphi)$.
\end{description}
\end{multicols}
\noindent Above, \ax{FS} stands for {\em Fischer Servi,} \ax{DP} for {\em disjunctive possibility,} \ax{RV} for {\em Rodríguez--Vidal}, and \ax{GD} for {\em G\"odel--Dummett}.

\begin{definition}
    We define the following extensions of $\ckf$:
    \begin{multicols}2
    \begin{description}
        \item[]
        \item[$\ikf$] $ := \ckf + \{\ax{FS}, \ax{DP} \}$,
        \item[$\kfi$] $  := \ckf + \{\ax{RV}, \ax{DP} \}$,
        \item[$\gkf$] $  := \ikf + \{\ax{GD}\}$,
        \item[$\gkfc$] $  := \gkf + \{\ax{RV}\}$,
        \item[$\csf$] $:= \ckf + \{ \ax{T_\Box}, \ax{T_\Diamond} \}$,
        \item[$\isf$] $ := \csf + \{\ax{FS}, \ax{DP} \}$,
        \item[$\sfi$] $  := \csf + \{\ax{RV}, \ax{DP} \}$,
        \item[$\gsf$] $  := \isf + \{\ax{GD}\}$,
        \item[$\gsfc$] $  := \gsf + \{\ax{RV}\}$.
                
    \end{description}
    \end{multicols}
    Above, $\sf c$ stands for `crisp' according to its use in fuzzy logic~\cite{RodriguezV21}.
    Let $\Lambda$ be any of the logics described above.
    We say $\Lambda$ is $\Diamond$-regular if $\ax{DP} \in \Lambda$.
    If $\Gamma$ is a set of formulas and $\varphi$ is a formula, we write $\Gamma\vdash_\Lambda\varphi$ if there are formulas $\psi_0,\dots, \psi_n$ such that $(\psi_0\land \cdots \wedge\psi_n)\to\varphi\in\Lambda$.
\end{definition}

\section{General Semantics}
\label{sec::general-semantics}

In this section, we describe generalized polyderivative semantics which generalize both the birelational Kripke semantics and the polytopological semantics we describe later.
This is particularly crucial for variations of $\KF$, since in this logic Kripke semantics is not a special case of topological semantics.

\begin{definition}
    If $X$ is a non-empty set, $d\colon\mathcal P(X) \to \mathcal P(X)$ is a {\em $T_d$-derivative operator} if, for all $A,B\subseteq X$:
    \begin{itemize}
        \item $d(\varnothing) = \varnothing$,
        \item $d(A\cup B) = d(A)\cup d(B)$,
        \item $d(d(A)) \subseteq d(A)$.
    \end{itemize}
    We say $d$ is a {\em closure} operator if moreover $A\subseteq d(A)$ for all $A\subseteq X$, and say that $A$ is {\em $d$-closed} if $d(A)\subseteq A$.
    Meanwhile, $e\colon\mathcal P(X) \to \mathcal P(X)$ is a {\em $T_d$-integral operator} if, for all $A,B\subseteq X$:
    \begin{itemize}
        \item $e(X) = X$,
        \item $e(A\cap B) = e(A) \cap e(B)$,
        \item $e(A) \subseteq e(e(A))$.
    \end{itemize} 
    We say $e$ is an {\em interior} operator if moreover $e(A)\subseteq A$ for all $A\subseteq X$, and say that $A$ is {\em $e$-open} if $A\subseteq e(A)$.
    
    Every derivative operator gives rise to an integral operator (and \emph{vice-versa}) by dualising in the usual way, and we define $d$-open and $e$-closed sets according to this dual.
    We typically denote closure and interior operators by the letters $c$ and $i$ instead of $d$ and $e$, respectively.
\end{definition}

Below, we refer to $T_d$-derivative and $T_d$-integral operators simply as ``derivative'' or ``integral'' operators.
We also omit parenthesis when not ambiguous, \emph{e.g.}, write $eA$ instead of $e(A)$.

\begin{definition}
    A polyderivative $\ckf$-frame is a tuple $\tuple{X, \inti, d_\Diamond, e_\Box}$, where $\inti$ is an interior operator, $d_\Diamond$ is a derivative operator, and $e_\Box$ is an integral operator, all over $X$.
    We let $\taui$, $\taud$, $\taub$ be the respective topologies (of open sets) and moreover demand that $e_\Box \inti = \inti e_\Box \inti $ and $e_\Box \inti = i_\Diamond e_\Box \inti $.
    
    We obtain polyderivative $\ckf$-models by adding a valuation function $\val{\cdot}$ taking propositional variables $p$ to open sets in $\taui$.
\end{definition}

\begin{definition}
    Let $\Model$ be a polyderivative $\ckf$-model.
    We extend $\val\cdot = \val\cdot^\Model$ to arbitrary formulas by structural induction as follows:
    \begin{multicols}{2}
    \begin{itemize}
        \item $\val{\bot}              = \emptyset$,
        \item $\val{\varphi\land\psi}  = \val{\varphi}  \cap \|\psi\|$,
        \item $\val{\Box\varphi}       = \intb\|\varphi\|$,
        \item $\val{\varphi\to\psi}    = \inti\big ((X\setminus\|\varphi\| ) \cup \|\psi\| \big)$,
        \item $\val{\varphi\lor\psi}   = \|\varphi\| \cup \|\psi\| $,
        \item $\val{\Diamond\varphi}   = \inti \clod \|\varphi\|$.
    \end{itemize}
    \end{multicols}
    \noindent If $x\in X$, we may write $\mathcal M,x\models\varphi$ instead of $x\in\val\varphi$; or simply $ x\models\varphi$, if $\mathcal M$ is clear from context.
\end{definition}

In practice, derivative and integral operators are typically derived from some additional structure: specifically, they may be induced from transitive relations or from topologies, as we will see in later sections.

\begin{definition}\label{defTopProp}
    Let $\mathcal {X} = \tuple{X, \inti, d_\Diamond, e_\Box}$ be a polyderivative $\ckf$-space.
    We say that ${\mathcal X}$ is:
    \begin{itemize}
        \item \emph{$\Diamond$-Coarse} iff $\inti \intd = \intb\inti$,
        \item \emph{$\Diamond$-Regular} iff $\clod\inti =  \inti\clod\inti$, 
        \item \emph{$\Box$-Regular} iff $\intd \inti = \intb\inti$, 
        \item \emph{extremally disconnected} iff $c_\to i_\to = i_\to$,
        \item \emph{hereditarily extremally disconnected} iff $\taui$ is each subspace of $\mathcal{X}$ is extremally disconnected.
    \end{itemize}
    If $\Model = \tuple{X, \taui, \taud, V}$, is a polyderivative $\ckf$-model, then each of the properties is inherited from $\tuple{X, \taui, \taud, V }$.
\end{definition}


It is known that certain confluence properties allow us to interpret modalities `classically'~\cite{balbiani2024variants}; the same holds for regularity properties.
The following is easy to check from the definitions, using the fact that $\val\varphi$ is always $\taui$-open.
\begin{lemma}
    \label{lemmTReg}
    Let $\Model$ be a polyderivative $\ckf$-model and $\varphi$ be any formula.
    \begin{enumerate}
        \item\label{itTreg1} If $\Model$ is $\Diamond$-regular, then $\val{\Diamond\varphi} = \clod \val\varphi$.
        \item\label{itTreg2} If $\Model$ is $\Box$-regular, then $\val{\Box\varphi} = \intd \val\varphi$.
    \end{enumerate}
\end{lemma}

Next, we use the bitopological regularity properties to define classes of spaces corresponding to some intuitionistic modal logics.
\begin{definition}\label{defTopLogic}
    If $\mathcal{X} = \tuple{X, \inti, d_\Diamond, e_\Box}$ is a polyderivative $\ckf$-space, we say that ${\mathcal X}$ is a(n):
    \begin{itemize}
        \item \emph{$\ikf$-space} iff it is $\Diamond$-coarse and $\Diamond$-regular,
        \item \emph{$\kfi$-space} iff it is $\Box$-regular and $\Diamond$-regular,
        \item \emph{$\gkf$-space} iff it is a hereditarily extremally disconnected $\ikf$-space,
        \item \emph{$\gkfc$-space} iff it is a $\Box$-regular $\gkf$-space.
    \end{itemize}
    We say $\mathcal{X}$ is a polyderivative $\csf$-space iff $d_\Diamond$ and $e_\Box$ are closure and interior operators, respectively.
    The definition of $\isf$-, $\sfi$-, $\gsf$-, and $\gsfc$-spaces is analogous. 

    Whenever $\Lambda$ is any of the logics defined above, we say that a polyderivative model $\Model$ is a {\em polyderivative $\Lambda$-model} if it is based on a $\Lambda$-space.
\end{definition}

\section{Soundness}
\label{sec::soundness}

Let us now see that the regularity properties we have identified lead to soundness for intuitionistic modal logics.
\begin{proposition}
    \label{prop::soundness-various}
    Any polyderivative $\ckf$-space $\mathcal{X}$ validates all intuitionistic tautologies, $\ax{K}_\Box$, $\ax{K}_\Diamond$, $\ax{4}_\Box$, $\ax{4}_\Diamond$, $\ax{Nec}$, and $\ax{MP}$.
    If $\mathcal{X}$ is also a $\csf$-space, then $\mathcal{X}$ also validates $\ax{T}_\Box$ and $\ax{T}_\Diamond$.
\end{proposition}
\begin{proof}
    Most of these are standard or straightforward; see for example \cite{mctarski}.
    
    For illustration,  we show the case for the axiom $\ax{4}_\Diamond$.
    First note that, for any set $A\subseteq X$, as $\inti(A) \subseteq A$ and $\clod$ distributes over conjunctions, $\clod(\inti(A)) \subseteq \clod(A \cup \inti(A)) =  \clod(A)$.
    We then have
    $
        \|\Diamond\Diamond\varphi\|  = \inti \clod \inti \clod(U)  
                                     \subseteq \inti \clod \clod(U)  
                                     = \inti \clod(U)  
                                     = \|\Diamond\varphi\|.  
    $
\end{proof}

The rest of the axioms we consider are sound only for certain classes of spaces.

\begin{proposition}
    \label{prop::soundness-fs}
    Let ${\mathcal X}$ be a polyderivative $\ckf$-space.  
    \begin{enumerate}
        \item \label{itDReg} If ${\mathcal X}$ is $\Diamond$-regular, then ${\mathcal X}$ validates $\ax{DP}$.
        \item\label{itDRegBC} If ${\mathcal X}$ is $\Diamond$-regular and $\Diamond$-coarse, then ${\mathcal X}$ validates $\ax{FS}$.\footnote{Compare with the case for birelational semantics: downward confluence is not sufficient to satisfy $\ax{FS}$, but in the presence of forward confluence it is. (See \cite{degroot2025ckik}.)}
        \item\label{itBoxR}  If ${\mathcal X}$ is $\Box$-regular and $\Diamond$-regular, then $\mathcal{X}$ validates $\ax{RV}$.
        \item\label{itHed} If ${\mathcal X}$ is hereditarily extremally disconnected, then ${\mathcal X}$ validates $\ax{GD}$.
    \end{enumerate}
\end{proposition}
\begin{proof}
    Write ${\mathcal X}=\tuple{X,\taui, d_\Diamond, e_\Box}$ and let $\val\cdot$ be any valuation on ${\mathcal X}$.
    \medskip
    
    \noindent\ref{itDReg}.
    Suppose that ${\mathcal X}$ is $\Diamond$-regular.
    By Lemma~\ref{lemmTReg}(\ref{itTreg1}), $\val{\Diamond\theta} = \clod\val\theta$ for any formula $\theta$.
    Then:
    \[        \|\Diamond(\varphi\lor\psi)\|
                = \clod(\|\varphi\|\cup \|\psi\|)
                 = \clod \|\varphi\|  \cup \clod \|\psi\|
                = \|\Diamond\varphi\lor\Diamond\psi\|,
    \]
    where the second equality follows from the properties of derivative operators.
    \medskip

    \noindent \ref{itDRegBC}.
    Suppose that $\mathcal{X}$ is $\Diamond$-regular and $\Diamond$-coarse. Let $x \in\|\Diamond\varphi \to \Box\psi\|$.
    The $\Diamond$-regularity of $\mathcal{X}$ implies $x\in \inti(W\setminus \clod(\|\varphi\|) \cup \intb(\|\psi\|))$.
    From the duality of derivative and integral operators along with the fact that $\tau_\Box \subseteq \tau_\Diamond$, we get $x\in \inti(\intd(W\setminus \|\varphi\|) \cup \intd(\|\psi\|))$.
    Note that, $A$, $e(A) = e(A) \cap e(A\cup B)$, the union of two sets obtained by integrals is a subset of the integral of union of the original sets.
    We thus get $x\in \inti(\intd((W\setminus \|\varphi\|) \cup \|\psi\|))$.
    As $\mathcal{X}$ is $\Diamond$-coarse, $x\in \intb(\inti((W\setminus \|\varphi\|) \cup \|\psi\|))$.
    Therefore $x\in \|\Box(\varphi\to\psi)\|$.
    \medskip

    \noindent\ref{itBoxR}.
    Suppose that ${\mathcal X}$ is $\Diamond$- and $\Box$-regular.
    By Lemma \ref{lemmTReg}, we get $\|\Box (\varphi\lor\psi)\| = \intb(\|\varphi\|\cup \|\psi\|)$.
    Let $x\in\intb(\|\varphi\|\cup \|\psi\|)$ and $x\not\in d_\Diamond(\|\psi\|)$, then $x\in e_\Diamond(W\setminus\|\psi\|)$.
    By $\Box$-regularity, we get $x\in \intb(\|\varphi\|\cup \|\psi\|) \cap e_\Box(W\setminus\|\psi\|)$ and so $x\in \intb(\|\varphi\|)$.
    Therefore $\intd(\|\varphi\|\cup \|\psi\|) \subseteq \intb \|\varphi\|  \cup \clod \|\psi\|$.
    By \ref{lemmTReg} again, we get $\|\Box (\varphi\lor\psi)\| = \|\Box\varphi\lor\Diamond\psi\|$.
    \medskip
    
    \noindent\ref{itHed}.
    By Bezhanishvili \emph{et al.} \cite{bezhanishvili2015locallylinear}, Proposition 3.1. \qedhere
    
\end{proof}

By Propositions~\ref{prop::soundness-various} and ~\ref{prop::soundness-fs}, we have:
\begin{lemma}
    \label{lem::soundness}
    Let $\Lambda \in \{\ckf,\ikf,\kfi, \gkf, \gkfc, \csf,\isf,\sfi, \gsf, \gsfc \}$.
    Then $\Lambda$ is sound for the class of polyderivative $\Lambda$-models.
\end{lemma}

\section{Birelational semantics}
\label{sec::birelational-semantics}

The more familiar semantics for intuitionistic modal logics arises from birelational structures (see e.g.~\cite{AlechinaMPR01,servi1977modal}.
Birelational semantics are incomparable to polytopological semantics, but they are in fact a special case of polyderivative semantics, provided that the modal relation is weakly transitive.
In this section, we make this relationship precise.

\begin{definition}
    \label{def::csf-model}
    A birelational $\ckf$-frame is a tuple $\mathcal F=\tuple{\domai \Frame ,\peq^\Frame, \rel^\Frame}$, where:
    \begin{itemize}
        \item $\domai \Frame$ is the set of \emph{possible worlds},
        \item the \emph{intuitionistic relation} $\peq^\Frame$ is a preorder (i.e., a reflexive and transitive relation) on $\domai \Frame$, and
        \item the \emph{modal relation} $\rel^\Frame$ is a transitive relation on $\domai \Frame$.
    \end{itemize} 
    A {\em $\ckf$-model} is a tuple $\Model = \tuple{\domai \Model,\peq^\Model,\rel^\Model,\val\cdot^\Model}$, consisting of a $\ckf$-frame equipped with a  {\em valuation} $\val\cdot^\Model :\mathrm{Prop}\to \mathcal{P}(\domai \Model)$ such that, for all $p\in \mathrm{Prop}$, if $w \peq v$ and $w\in \val p^\Model$, then $v\in \val p^\Model$.
\end{definition}

When clear from context, we may drop superscripts and write e.g.~$\peq$ instead of $\peq^\Frame$.
Note we stray from the familiar definition of $\csf$-models in that we require no interaction between $\peq$ and $\rel$.
Instead, we will make sure that our models validate $\csf$ by modifying the semantics of $\Box$.
Below, if $R\subseteq A\times B$ and $S\subseteq B\times C$, then we define $R;S \subseteq A\times C$ by $a\mathrel{(R;S)}c$ iff there is $b\in B$ such that $a\mathrel Rb\mathrel S c $; and, if $B=A$, then $R^+$ denotes the transitive closure of $R$.

If $R\subseteq X\times X$, we define $d_R \colon\mathcal P(X) \to \mathcal P(X)$ by $d_R(A) = R^{-1} A$, and $e_R(A) = X\setminus d_R(X\setminus A)$.

Fix a $\ckf$-frame $\Frame$ and define ${\brel^\Frame}:=  {(\peq^\Frame;\rel^\Frame)^+}$.
We associate with $\Frame$ a $\ckf$-space $\ms{\Frame } := \tuple{\domai \Frame,d_\peq,d_\rel,e_\brel} $.
We tacitly identify $\Frame$ with this $\csf$-space, so that for example $\Frame\models\varphi$ should be understood as $\ms{\Frame } \models\varphi $.
If $\Model$ is such that its underlying space is of the form $\ms{\Frame}$, we say that $\Model$ is a {\em birelational $\csf$-model.}

\begin{lemma}
    If $\Model$ is a birelational $\ckf$-model, then:
    \begin{enumerate}
        \item $\mathcal M ,w\models \varphi\to\psi$ iff, for all $v\in W$, if $w\peq v$ and $\mathcal M ,v\models\varphi$, then $\mathcal M ,v\models\psi$,
        \item $\mathcal M ,w\models \Box\varphi$ iff, for all $v\bler w$,  $\mathcal M ,v\models\varphi$, and
        \item $\mathcal M ,w\models \Diamond\varphi$ iff, for all  $v\seq w$,  there is $u \ler v$ such that $\mathcal M ,u\models\varphi$.
    \end{enumerate}
\end{lemma}
 
Frames for the other logics we consider are obtained by enforcing various `confluence' properties.

\begin{definition}
    \label{def::birelational-properties}
    We say that a $\ckf$-frame $\Frame$ is:
    \begin{itemize}
        \item \emph{forward confluent} iff $w\rel v$ and $w\peq w'$ implies there is $v'$ such that $v\peq w' \rel v'$,
        \item \emph{backward confluent} iff $w \rel v \peq v'$ implies there is $w'$ such that $w \peq w' \rel v'$,
        \item \emph{downward confluent} iff $w \peq  v \rel v'$ implies there is $w'$ such that $w \rel w' \peq v'$,
        \item \emph{locally linear} iff $w\peq v$ and $w\peq u$ implies that either $v\peq u$ or $u\peq v$.
    \end{itemize}
    Forward, backward, and downward confluence are illustrated in Figure \ref{figure::confluences}.
\end{definition}

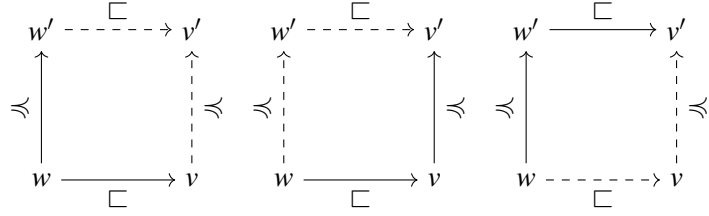
\begin{figure}[ht]
\centering
\tikzstyle{world}=[circle,draw,minimum size=5mm,inner sep=0pt]
\begin{tikzpicture}
    \node (w) at (0,-2) {$w$};
    \node (w2) at (0,0) {$w'$};
    \node (v) at (2,-2) {$v$};
    \node (v2) at (2,0) {$v'$};

    \draw[->] (w) -- (w2) node[midway,left] {$\peq$};
    \draw[->] (w) -- (v) node[midway,below] {$\rel$};
    \draw[dashed,->] (v) -- (v2) node[midway,right] {$\peq$};
    \draw[dashed,->] (w2) -- (v2) node[midway,above] {$\rel$};
\end{tikzpicture}
\begin{tikzpicture}
    \node (w) at (0,-2) {$w$};
    \node (w2) at (0,0) {$w'$};
    \node (v) at (2,-2) {$v$};
    \node (v2) at (2,0) {$v'$};

    \draw[dashed,->] (w) -- (w2) node[midway,left] {$\peq$};
    \draw[->] (w) -- (v) node[midway,below] {$\rel$};
    \draw[->] (v) -- (v2) node[midway,right] {$\peq$};
    \draw[dashed,->] (w2) -- (v2) node[midway,above] {$\rel$};
\end{tikzpicture}
\begin{tikzpicture}
    \node (w) at (0,-2) {$w$};
    \node (w2) at (0,0) {$w'$};
    \node (v) at (2,-2) {$v$};
    \node (v2) at (2,0) {$v'$};

    \draw[->] (w) -- (w2) node[midway,left] {$\peq$};
    \draw[dashed,->] (w) -- (v) node[midway,below] {$\rel$};
    \draw[dashed,->] (v) -- (v2) node[midway,right] {$\peq$};
    \draw[->] (w2) -- (v2) node[midway,above] {$\rel$};
\end{tikzpicture}
\caption{Schematics for forward (left), backward (center), and downward (right) confluence. Solid arrows correspond to the universal quantifiers and dashed arrows correspond to the existential quantifiers.}
\label{figure::confluences}
\end{figure}

Semantics for $\csf$ is typically based on backward confluent frames. 
However, we do not make such a requirement, and instead modify the semantics for $\Box$ to use the relation ${\brel} = {(\peq;\rel)^+}$ instead of $\peq;\rel$.
The main reason for using a different approach is that it will generalize nicely to topological semantics, but for relational structures the difference is inessential due to the two relations being equal on such frames.
On downward confluent frames, we instead have ${\brel}={\rel; \peq}$

\begin{lemma}\label{lemmTC}
    Let $\Frame$ be a $\ckf$-frame.
    \begin{enumerate}
        \item\label{itTC1} If $\Frame$ is backward confluent, then ${\brel^\Frame} = {\peq^\Frame; \rel^\Frame}$.
        \item\label{itTC2} If $\Frame$ is downward confluent, then ${\brel^\Frame} =  {\rel^\Frame; \peq^\Frame}$.
    \end{enumerate}
\end{lemma}

\begin{proof}
    Omitting superscripts, for the first item, it suffices to check that $\peq ; \rel $ is already transitive.
    Suppose that $w\mathrel{(\peq ; \rel )^2} v$, so that there are $w_1,u_1,u_2$ with $w  \peq w_1\rel u_1\peq u_2 \rel v$.
    By backward confluence, there is $w_2$ such that $w_1\peq w_2 \rel u_2$.
    By transitivity of both relations, $w\peq u_2\rel v$, hence $w \mathrel{(\peq;\rel)}v$.
    
    The second item is similar, using downward confluence in place of backward confluence.
\end{proof}

\begin{definition}
    Let $\Frame = \tuple{\domai \Frame ,\peq^\Frame, \rel^\Frame}$ be a birelational $\ckf$-frame.
    Then, $\Frame$ is a(n):
    \begin{itemize}
        \item \emph{$\ikf$-frame} iff it is forward and backward confluent,\footnote{Despite our modified semantics for $\Box$, backward confluence is necessary here. On the model $w \sqsubset v\peq u$, with $p$ holding only on $u$ and $q$ holding nowhere; we have $w\models\Diamond p \to \Box q$ but $w\not\models\Box(p\to q)$.}
        \item \emph{$\kfi$-frame} iff it is forward and downward confluent,
        \item \emph{$\gkf$-frame} iff it is locally linear,
        \item \emph{$\gkfc$-frame} iff it is a downward confluent $\gsf$-frame.
    \end{itemize}
    We say $\Frame$ is a $\csf$-frame if $\rel^\Frame$ is transitive.
    The definition of $\isf$-, $\sfi$-, $\gsf$-, and $\gsfc$-spaces is analogous, instead assuming that $\Frame$ is a birelational $\csf$-frame. 
    
    A $\csf$-model inherits the confluence properties of its underlying frame, \emph{e.g.,}~a model $\mathcal M $ is forward confluent iff the frame $\tuple{\domai \Model, \peq^\Model, \rel^\Model }$ is.
    In particular, if $\Lambda$ is any of the logics described above, we say that $\mathcal M$ is a {\em $\Lambda$-model} iff  $\tuple{\domai \Model, \peq^\Model, \rel^\Model }$ is a $\Lambda$-frame.
    If $\varphi$ is a formula, we write $\Lambda\models\varphi$ if $\mathcal M,w\models\varphi$ for every $\Lambda$-model $\mathcal M$ and $w\in \domai \Model$.
\end{definition}

\begin{theorem}
    \label{thm::completeness-all}
    Let $\Lambda \in \{\ckf,\csf,\ikf,\kfi,\gkf,\gkfc, \isf,\sfi,\gsf,\gsfc\}$. 
    Then $\Lambda$ is sound and strongly complete with respect to the class of birelational $\Lambda$-models.
\end{theorem}

These results have either appeared in the literature or are minor variations~\cite{AlechinaMPR01,balbiani2024variants,servi1984axiomatizations,Simpson94}, but we provide a detailed proof sketch in the technical appendix.

\section{Topological Completeness}
\label{sec::topological-completeness}

Relational semantics for $\SF$ (and, by extension, to intuitionistic logic) may be seen as a special case of topological semantics, but this is no longer the case for $\KF$, where in general the Cantor derivative induced by a transitive frame only coincides with its relational semantics when the frame is irreflexive~\cite{beg}.
Topological completeness can still be deduced from birelational completeness, but some post-processing is needed.
In this section, we introduce polytopological models and show how the birelational completeness results can be lifted to this setting.

\begin{definition}
    \label{def::bitopological-space}
    A  {\em tritopological space} is a tuple ${\mathcal X} = \tuple{\domai {\mathcal X},\taui^{\mathcal X},\taud^{\mathcal X},\taub^{\mathcal X}}$, where:
    \begin{itemize}
        \item $\domai {\mathcal X}$ is a non-empty set, and
        \item $\taui^{\mathcal X}$, $\taud^{\mathcal X}$, and $\taub^{\mathcal X}$ are topologies over $\domai {\mathcal X}$.
    \end{itemize}    
    A {\em bitopological space} is defined similarly, but omitting $\taub^{\mathcal X}$.
    
    For a connective $\circ\in\{i,\Diamond,\Box\}$, we define $i^{\mathcal X}_\circ$ and $c^{\mathcal X}_\circ$ to be the respective interior and closure operators with respect to $\tau^{\mathcal X}_\circ$.
    Similarly, we let $d_\circ$ and $e_\circ$ denote the Cantor derivative and integral.   

    Every tritopological space ${\mathcal X}=\tuple{\domai {\mathcal X},\taui^{\mathcal X},\taud^{\mathcal X}, \taub^{\mathcal X}}$ gives rise to two {\em induced} polyderivative spaces: namely, $\mc{{\mathcal X}} := \tuple{\domai {\mathcal X} ,\inti^{\mathcal X} ,i_\Diamond^{\mathcal X} ,c_\Box^{\mathcal X}  }$ and $\md{{\mathcal X}} :=\tuple{\domai {\mathcal X} ,\inti^{\mathcal X} ,d_\Diamond^{\mathcal X} ,e_\Box^{\mathcal X}  }$.
    Spaces of the first form are {\em tritopological closure spaces} and of the second, {\em tritopological   derivative spaces}.

    A {\em tritopological closure/derivative model} is a tuple $\mathcal M = \tuple{\domai {\Model},\inti^\Model,d_\Diamond^\Model,e_\Box^ \Model ,\val\cdot^\Model}$, consisting of a polyderivative space of the respective form equipped with a valuation $\val\cdot^\Model\colon\mathrm{Prop} \to \taui^\Model$.
    Again, we drop the superscripts whenever possible.
    
    For $\Lambda\in \{\csf,\isf,\sfi, \gsf,\gsfc\}$, a {\em tritopological $\Lambda$-space} is a tritopological closure space which is also a $\Lambda$-space.
    Similarly, for $\Lambda\in \{\ckf,\ikf,\kfi, \gkf,\gkfc\}$, a {\em tritopological $\Lambda$-space} is a is a tritopological derivative space which is also a $\Lambda$-space.
\end{definition}

Note that, if $\mathcal X$ is a $\Lambda$-space for any of the $\KF$ logics $\Lambda$ we consider, then $\taud^{\mathcal X}$ and $\taub^{\mathcal X}$ are $T_d$, as they are assumed to validate the $\ax{4}_\Box$ and $\ax{4}_\Diamond$ axioms.

Often, we want to work with bitopological rather than tritopological spaces, as these are a bit simpler.
Bitopological spaces also induce polyderivative spaces, by first letting them induce a tritopological space.
To this end, note that if $X$ is any set and $\tau,\sigma$ are topologies on  $X$, then $\tau\cap\sigma$ is also a topology on $X$ and is clearly their meet, i.e.~the greatest lower bound of $\{\tau,\sigma\}$ under inclusion.
Every bitopological space ${\mathcal X}=\tuple{\domai {\mathcal X},\taui^{\mathcal X},\taud^{\mathcal X} }$ thus induces a tri-topological space by adding $\taub^{\mathcal X} := \taui^{\mathcal X}\cap \taud^{\mathcal X}$. 
We then say that the {\em induced} spaces of the bitopological space $\mathcal X$ are the induced spaces of its induced tritopological space.
Induced spaces are {\em bitopological closure spaces} and {\em bitopological derivative spaces,} respectively.

Note, however, that in the general case, the third topology need not necessarily be the meet of the first two, and indeed Cantor derivative logics are not complete for spaces of this form.

\begin{lemma}\label{lemmInducedIncomplete}
If $\mathcal X$ is a  derivative space induced from a tritopological space $\tuple{X,\taui,\taud,\taub}$ where $\taub\subseteq \taui $, then
\[\mathcal X \models \Box q\to p\vee (p\to q).\]
This formula is not derivable in $\gkfc$ (or any other $\KF$ logic we consider).
\end{lemma}

\begin{proof}
    Let $\val\cdot$ be a valuation on $\mathcal X$, $x\in X$, and suppose that $x\in \val{\Box q} \setminus \val p$.
    Then, $x$ has a $\taub$-neighbourhood $U$ such that $U\setminus \{x\} \subseteq \val{q} $.
    By assumption, $U$ is also a $\taui$-neighbourhood of $x$, and since $x\notin\val p$, $U \cap \val { p} \subseteq U\setminus \{x\} \subseteq \val q$, hence $U$ witnesses that $x\in\val{p\to q}$.
    Since $x$ was arbitrary, $\val {\Box q\to p\vee (p\to q)}  = X$.
    
    For a counterexample, let $X = \{0,1,2\} $ with $0\peq 1$  and $0,1\rel 2$ (with $\peq$ being reflexive), and let $\val p=\{1\}$ and $\val q =\{2\}$.
    Then, it is not hard to check that $0\notin \val {\Box q\to p\vee (p\to q)} $.
\end{proof}

Thus we cannot avoid working with tritopological spaces if we expect to prove completeness for our target logics.

Although topological semantics are not, strictly speaking, a generalization of transitive relational semantics, it is possible to assign a topological space to a binary relation in a way that oftentimes does preserve modal semantics.
Given a binary relation $R\subseteq W\times W$, and $w\in W$, we define ${\uparrow_R} w  = \{v \mid w\mathrel R v\}$.
We can define a topology $\tau_R$ on $W$ where $U\in\tau_R$ is open if and only if ${\uparrow_R} U\subseteq U $, where ${\uparrow_R} U = \bigcup_{u\in U}{\uparrow_R} u $.
If $R$ is a preorder, then the sets of the form $\uparrow_R w$ form a basis for $\tau_R$.
Note that $R$ does not need to be a preorder for $\tau_R$ to be a topology, but in this case $\uparrow_R w$ is not always open.
However, we can replace $R$ by its transitive, reflexive closure and obtain the same topology: more general, if $R\subseteq S\subseteq R^*$, then $\tau_R = \tau_S = \tau_{R^*}$.
It is moreover readily checked that $\tau_{R\cup S} = \tau_{R}\cap \tau_S$, and if $R,S$ are both reflexive then $R\cup S \subseteq (R;S)^+\subseteq (R\cup S)^*$, so in this context $\tau_{R}\cap \tau_S = \tau_{(R;S)^+}$.
Comparing this with the definition of $\brel:=(\peq;\rel)^+$, we obtain the following characterization of

\begin{lemma}\label{lemmTrirelToBitop}
    If $\tuple{X,\peq,\rel}$ is a reflexive birelational frame then $\tau_{\brel} = \tau_{\peq}\cap \tau_{\rel} $.
\end{lemma}

Note that reflexivity is essential here, as can be deduced from Lemma~\ref{lemmInducedIncomplete}, which shows that using the meet topology leads to validity of formulas not valid using $\tau_{\brel}$.
Topological spaces of the form $(W,\tau_R)$ are {\em Alexandroff spaces,} and are characterized by the property that arbitrary intersections of open sets are open.
We may use such spaces to obtain several topological completeness results from their Kripke semantics counterparts.

\begin{definition}
    If $\mathcal F =  \tuple{W,   \peq, \rel  }$ is a birelational $\ckf$-frame, its associated tritopological space is $\mt{\Frame} = \tuple{W,   \tau_\peq, \tau_\rel, \tau_\brel}$.
\end{definition}


\begin{proposition}
    \label{prop::topologizing-preserves-truth}
    Let $\Frame = \tuple{W,   \peq, \rel  }$ be a birelational $\ckf$-frame.
    \begin{enumerate}
        \item If $\rel$ is reflexive, then the polyderivative and bitopological closure spaces associated with $\mathcal{F}$ coincide, \emph{i.e.}, $\ms{\Frame} = \mc{(\mt{\Frame})}$.
        \item If $\rel$ is irreflexive, then the polyderivative and tritopological derivative spaces associated with $\mathcal{F}$ coincide, \emph{i.e.}, $\ms{\Frame} = \md{(\mt{\Frame})}$.
    \end{enumerate}
\end{proposition}

\begin{proof}
    This is shown for a single relation in, \emph{e.g.}~\cite{BaltagBF23}.
    It is straightforward to adapt their argument to the non-classical case with three relations.
    In the first item, we can moreover replace tritopological semantics by bitopological semantics using Lemma~\ref{lemmTrirelToBitop}.
\end{proof}

\begin{proposition}
    \label{prop::classical-diamonds}
    Let $\Frame = \tuple{W ,\peq, \rel}$ be a reflexive or irreflexive $\ckf$-frame.
    \begin{enumerate}
        \item\label{itClassDBack} If $\Frame$ is backward confluent, then $\mt{\Frame}$ is $\Diamond$-coarse.
        \item\label{itClassDFor} If $\Frame$ is forward confluent, then $\mt{\Frame}$ is $\Diamond$-regular.
        \item\label{itClassDDwn} If $\Frame$ is downward confluent, then $\mt{\Frame}$ is $\Box$-regular.
        \item\label{itLocallyLinear} If $\peq$ is locally linear, then $\tau_\peq$ is hereditarily extremally disconnected.
    \end{enumerate}
\end{proposition}

\begin{proof}
    In this proof, we make tacit use of Proposition \ref{prop::topologizing-preserves-truth} and identify operators in the topological setting with the operators obtained from birelational models in Section \ref{sec::birelational-semantics}, \emph{e.g.}, $i_\to(X) = \{w\in W\mid \forall v\in w. w \peq v \text{ implies }v\in X\}$.
    \medskip
    
    \noindent\ref{itClassDBack}. 
    Assume that $\Frame$ is backward confluent and let $X\subseteq W$.
    Let $w\in \inti\intd(X)$ and $u,v\in W$ be such that $w\rel u \peq v$.
    By backward confluence, there is $s$ such that $w\peq s\rel u$.
    From $w\in \inti\intd(X)$ and $w\peq s\rel u$, we get $u\in X$.
    As the argument above work for any $v$ and $u$, $w\in \intb\inti(X)$.
    On the other hand, let $w\in \intb\inti(X)$ and $u,v\in W$ be such $w\peq u \rel v$.
    Then $w\brel v \rel v$, and so $v\in X$.
    We conclude $w\in \inti\intd(X)$.
    \medskip
    
    \noindent\ref{itClassDFor}.
    Assume that $\Frame$ is forward confluent and let $X\subseteq W$.
    It is immediate from the definition of interior operators that $\inti\clod\inti(X) \subseteq \clod\inti(X)$.
    Now, suppose $w\in \clod\inti(X)$, so there is $v$ such that $w\rel v$ and $v\in \inti(X)$.
    If $w\peq s$, then there is $u$ such that $s\rel u$ and $v\peq u$ by forward confluence.
    As $\inti(X)$ is $\tau_\peq$-open, $u \in \inti(X)$.
    Therefore $s\in \clod\inti(X)$.
    We conclude $w\in \inti\clod\inti(X)$.
    \medskip
    
    \noindent\ref{itClassDDwn}.
    Assume that $\Frame$ is downward confluent and let $X\subseteq W$.
    Let $w\in \intd\inti(X)$ and $u,v\in W$ be such that $w\peq u \brel v$.
    By downward confluence, there is $s$ such that $w\rel s\peq u$.
    From $w\in \intd\inti(X)$ and $w\rel s\peq u$, we get $u\in X$.
    As the argument above work for any $v$ and $u$, $w\in \inti\intb(X)$.
    On the other hand, let $w\in \inti\intb(X)$ and $u,v\in W$ be such $w\rel u \peq v$.
    Then $w\brel v \peq v$, and so $v\in X$.
    We conclude $w\in \intb\inti(X)$.
    
    \medskip
    
    \noindent\ref{itLocallyLinear}. By Proposition 3.1 of Bezhanishvili \emph{et al.} \cite{bezhanishvili2015locallylinear}.
\end{proof}

\begin{theorem}
    \label{thm::completeness}
    If $\Lambda\in \{\csf,\isf,\sfi, \gsf,\gsfc\}$,
    then $\Lambda$ is strongly complete for the class of Alexandroff bitopological $\Lambda$-models.
    Furthermore, the finite model property holds for this class of $\Lambda$-models.
\end{theorem}

\begin{proof}
    By Propositions~\ref{prop::topologizing-preserves-truth} and \ref{prop::classical-diamonds} along with Theorem~\ref{thm::completeness-all}.
    The finite model property hold for $\Lambda$ as it holds for birelational $\Lambda$-models \cite{balbiani2024variants, girlando}.
\end{proof}

\begin{theorem}
    \label{thm::completeness-k4}
    If $\Lambda\in \{\ckf,\ikf,\kfi, \gkf,\gkfc\}$,
    then $\Lambda$ is strongly complete for the class of Alexandroff tritopological $\Lambda$-models.
\end{theorem}
\begin{proof}
    We show that each birelational $\Lambda$-model is bisimilar to an irreflexive $\Lambda$-model.

    We say $\Model = \tuple{W, \peq, \sqsubset,V}$ and $\Model' = \tuple{W', \peq', \sqsubset',V'}$ are bisimilar if there is a non-empty relation $B \subseteq W\times W'$ which is a bisimulation (as in the single relation setting) with respect to both intuitionistic and modal relations.
    By essentially the same proof as the single relation case, if $wBw'$, then $M,w\models\varphi$ iff $M',w'\models\varphi$, for all formula $\varphi$.

    Let $\Model = \tuple{W, \peq, \sqsubset,V}$ be a $\Lambda$-model.
    We define $\Model^\mathrm{irr} = \tuple{W^\mathrm{irr}, \peq^\mathrm{irr}, \sqsubset^\mathrm{irr},V^\mathrm{irr}}$ as follows:
    \begin{itemize}
        \item $W^\mathrm{irr} = \{\tuple{w,n} \mid w\in W, n\in\mathbb{Z}\}$,
        \item $\tuple{w,n} \peq^\mathrm{irr}\tuple{v,m}$ iff $w\peq v$,
        \item $\tuple{w,n} \rel^\mathrm{irr}\tuple{v,m}$ iff $w\sqsubset v$ and $n<m$, and
        \item $\tuple{w,n} \in V(p)$ iff $w\in V(p)$.
    \end{itemize}

    It is straightforward to check that $\Model^\mathrm{irr}$ is a $\Lambda$-model.
    In particular, the transitiveness of $\peq^\mathrm{irr}$ and $\rel^\mathrm{irr}$ follow from the transitiveness of $\peq$ and $\rel$ respectively, along with the transitiveness of the usual ordering on the integers in the latter.
    The confluence properties of $\Model^\mathrm{irr}$ follow from the confluence properties of $\Model$.
    For example, if $\Model$ is forward confluent, $\tuple{w,m}\peq^\mathrm{irr} \tuple{v,n}$ and $\tuple{w,m} \rel^\mathrm{irr} \tuple{u,k}$, then there are $s$ such that $v\rel s$ and $u\peq s$.
    If $l = \max\{n,k\} + 1$, then $\tuple{v,n}\rel^\mathrm{irr} \tuple{s,l}$ and $\tuple{u,k}\peq^\mathrm{irr} \tuple{s,l}$.
    And so $\Model^\mathrm{irr}$ is forward confluent too.

    At last, we note that $\sqsubset^\mathrm{irr}$ is irreflexive as the ordering $<$ on integers is irreflexive.
    So $\Model^\mathrm{irr}$ is an irreflexive $\Lambda$-model.
    
    It is immediate that if we set $B = \{ \tuple{w,\tuple{v,n}} \mid w = v\}$, then $\Model$ and $\Model^\mathrm{irr}$ are bisimilar \emph{via} $B$.
    The theorem will then follow from Propositions~\ref{prop::topologizing-preserves-truth} and \ref{prop::classical-diamonds} along with Theorem~\ref{thm::completeness-all}.
\end{proof}

\begin{proposition}
    If $\Lambda\in \{\ckf,\ikf,\kfi, \gkf,\gkfc\}$,
    then $\Lambda$ does not have the finite model property with respect to Alexandroff tritopological $\Lambda$-models.
\end{proposition}
\begin{proof}[Proof sketch]
    On one hand, it is straightforward to show that finite $T_d$ (or $T_0$) topological spaces are scattered and thus that they satisfy Löb's axiom $\ax{L} := \Box(\Box p \to p)\to \Box p$. 
    On the other hand, $\ckf$, $\ikf$, $\kfi$, $\gkf$, and $\gkfc$ do not satisfy $\ax{L}$.
\end{proof}

\section{Conclusion}
\label{sec::conclusion}

We have developed a polytopological (and, more generally, polyderivative) framework in which to interpret all prominent `transitive' intuitionistic modal logics, as well as a few that had not been considered previously.
The most salient restriction of our work is that we have assumed all spaces to be $T_d$ (i.e., satisfy the $\KF$ axioms), but it is well known that the logic of Cantor derivative over all topological spaces is $\sf wK4$, which weakens the \ax{4} axiom to $p\wedge \Box p\to\Box\Box p$, and its dual, $\Diamond \Diamond p\to p\vee\Diamond p$.
The reason for working with such a restriction is that these axioms do not seem to yield completeness for the intended semantics.
Determining the logics obtained by dropping the $T_d$ assumption from each of the classes of models we consider immediately leads to several interesting problems we leave open.

We also note that, while our intuitionistic $\SF$ logics are complete for the class of bitopological deriative spaces, completeness over this class does not hold for $\KF$ logics in view of Lemma~\ref{lemmInducedIncomplete}.
Logics of bitopological derivative spaces may be of independent interest and a concrete family of questions is whether adding $\Box q\to p\vee (p\to q)$ to each of the $\KF$ logics we consider yields completeness for its bitopological derivative models.
We also suspect that bitopological versions of $\sf wK4$ may circumvent some of the technical issues with its tritopological variants.

Finally, we leave open the question of establishing completeness for `natural' topological spaces.
Our completeness results rely on birelational completeness, and hence produce Aleksandroff spaces, which usually violate most (if not all) of the separation axioms typically assumed of topological spaces.
One may wonder if completeness still holds for each logic when one or more of the topologies is Hausdorff or even metrizable, in the spirit of the McKinsey--Tarski theorem~\cite{mctarski}.
While any bitopological space induces a $\csf$-space and every $T_d$ bitopological space induces a $\ckf$-space, even the construction of Hausdorff $\isf$ spaces seems to be non-trivial.

\bibliographystyle{eptcs}
\bibliography{biblio}

\Appendix

\section{Strong completeness}
\label{sec::strong-completeness}

In this Appendix, we sketch the proof of strong completeness for $\ckf$, $\ikf$, $\kfi$, $\gkf$, $\gkfc$, $\csf$, $\isf$, $\sfi$, $\gsf$, $\gsfc$ with respect to their birelational semantics.
We do need separate arguments depending if the logic includes the axiom $\ax{DP}$.
For most of these logics, these results are already known:
see \cite{servi1984axiomatizations} and~\cite{Simpson94} for $\ikf$ and $\isf$; see \cite{balbiani2024variants} for the other variants of $\SF$.
As far as we are aware, completeness for $\ckf$, $\kfi$, $\gkf$, and $\gkfc$ have not appeared before in the literature.
Since our proofs are variations of those appearing in said references, we omit technical details which have appeared beforehand.
We particularly follow the presentation of~\cite{balbiani2024variants}.

It will be convenient to observe that the following formulas are contained in $\mathsf{CK4}$, and thus also its extensions.
We leave the proofs to the reader.
\begin{proposition}
    \label{auxiliary}
    The following formulas are provable in ${\mathsf{CS4}}$:
    \begin{multicols}{2}
    \begin{enumerate}[label=(\arabic*)]
    	\item\label{der:ax1} $\Diamond {\left( p \rightarrow q\right)} \rightarrow \left(\nec p \rightarrow \Diamond q\right)$,\footnote{This is the first of Fischer Servi's two \emph{connecting axioms} \cite{servi1984axiomatizations}.}	
    	\item \label{der:ax3} $\Diamond \left( p \wedge q\right) \rightarrow \Diamond p \wedge \Diamond q$,
    	\item\label{der:box_dist_and} \mbox{$\nec p \wedge \nec q  \rightarrow \nec \left(p \wedge q\right)$},
    	\item\label{der:ax2} \mbox{$\nec p \vee \nec q  \rightarrow \nec \left(p \vee q\right)$}.
    \end{enumerate}
    \end{multicols}
\end{proposition}

\subsection{Prime theories}

We first introduce the notions common to all our completeness proofs. 
Let $\Lambda$ be any of the logics defined in Section \ref{sec::syntax}.
A set $\Phi$ of $\lanfull$-formulas is called a \define{$\Lambda$-theory} if it is closed under syntactic consequence, \emph{i.e.}, $\Phi \vdash_\Lambda \varphi$ implies $\varphi \in  \Phi$.
Furthermore, we say that $\Phi$ is \define{prime} if $\varphi \vee \psi \in  \Phi$ implies that either $\varphi \in  \Phi$ or $ \psi \in  \Gamma$.
When clear from context, we omit the reference to $\Lambda$ and refer to $\Lambda$-theories as theories.

\begin{definition}
    \label{def:deduction} 
	Given sets $\Phi$ and $\Xi$ of formulas, we say that $\Phi$ is \define{$\Xi$-consistent} if 	$\Phi  \not \vdash   \Xi$, and $\Phi$ is \define{maximal $\Xi$-consistent} if moreover $\Phi\subsetneq \Phi'$ implies that $\Phi'$ is not $\Xi$-consistent.
    We say that $\Phi$ is \define{consistent} if it is $\varnothing$-consistent (with the understanding that $\bigvee\varnothing \equiv \bot$).
    If $\Xi$ is a singleton $\{\xi\}$, we write $\xi$-consistent instead of $\{\xi\}$-consistent.
    Maximal consistency and maximal $\xi$-consistency are defined analogously.
\end{definition}

Let $\Phi$ and $\Psi$ be theories.
Note that, if $\Phi$ is $\Xi$-consistent, then necessarily $\bot \not\in\Phi$.
We say that the theory $\Psi$ \define{extends} $\Phi$ if $\Phi \subseteq\Psi$.

The following Lindembaum Lemma will be used below:
\begin{lemma}
    \label{lem:lindenbaum} 
    Any $\Xi$-consistent set $\Phi$ of formulas can be extended to a maximal $\Xi$-consistent prime theory $\Phi_*$.
\end{lemma}	
\begin{proof}
    This is a standard argument: enumerate all formulas $\{\phi_i : i \in\mathbb{N}\}$ in the language and by induction on $i$ add $\phi_i$ to the set if the resulting theory is $\Xi$-consistent. 
    Do so in such a way that whenever a disjunction is added, one of the disjuncts is added too. 
    It is straightforward to check that this is always possible.
    The $\Lambda$-theory generated by $\Phi$ and the formulas thus added is then maximal $\Xi$-consistent.
\end{proof}

\subsection{Completeness of \texorpdfstring{$\ps$-regular}{diamond-regular} logics}\label{secCompGS4}
In this subsection, suppose that $\Lambda$ is $\Diamond$-regular, that is, $\ax{DP}\in \Lambda$.
We will uniformly define the canonical model for any $\ps$-regular logic $\Lambda$.
Furthermore, if $\Phi$ is a $\Lambda$-theory then we define $\bfrm\Phi \coloneqq \{ \varphi \in \lanfull \mid \nec\varphi\in \Phi  \}$ and $\dfrm\Phi \coloneqq \{ \varphi \in \lanfull \mid \ps\varphi\notin \Phi  \}$.

\begin{definition}
    Let $\Lambda$ be a $\ps$-regular logic.
    The canonical model for $\Lambda$ is $\mathcal M_{\mathrm c} ^\Lambda=(W_{\mathrm c} ,{\peq_{\mathrm c} },{\rel_{\mathrm c} }, V_{\mathrm c} )$, where:
    \begin{enumerate}[label=\alph*)]
    	\item $W_{\mathrm c} $ is the set of consistent, prime $\Lambda$-theories,
        \item 	${\peq_{\mathrm c}} \subseteq W_{\mathrm c} \times W_{\mathrm c} $ is defined by $\Phi \peq_{\mathrm c}  \Psi$ if and only if $\Phi \subseteq \Psi $,
        \item ${\rel_{\mathrm c}}  \subseteq W_{\mathrm c} \times W_{\mathrm c} $ is defined by $\Phi \rel_{\mathrm c}  \Psi$ if and only if $  \bfrm\Phi \subseteq  \Psi $ and $  \dfrm\Phi \cap \Psi = \varnothing $, 
        
        \item $V_{\mathrm c} \from \mathbb P \to 2^{W_{\mathrm c}}$ is defined by $V_{\mathrm c}(p):= \{\Phi\in W_{\mathrm c} \mid p\in \Phi\}$.
    \end{enumerate}
\end{definition}

This construction always yields $\ps$-regular models, given the assumption that $\Lambda$ is $\ps$-regular.

\begin{lemma}
    \label{lem::canonical-model-confluences}
    Let $\Lambda$ be any $\ps$-regular logic.
    \begin{enumerate}
        \item If $\ax{FS} \in \Lambda$, then $\mathcal M^\Lambda_{\mathrm c}$ is backward confluent.
        \item If $\ax{RV} \in \Lambda$, then $\mathcal M^\Lambda_{\mathrm c}$ is downward confluent.
        \item If $\ax{GD} \in \Lambda$, then $\mathcal M_{\mathrm c}^{\Lambda}$ is locally linear.
    \end{enumerate}
\end{lemma}
\begin{proof}
    We consider the first item explicitly; the others are similar.
    Recall that $\ax{FS}$ is $(\Diamond\varphi \to \Box\psi) \to \Box(\varphi\to\psi)$. 
    Suppose that $\Phi \sqsubset_{\mathrm c} \Psi {\peq_{\mathrm c}} \Psi'$. 
    Let $\Upsilon = \Phi \cup \{\Diamond\chi: \chi \in \Psi'\}$. 
    We argue that $\Upsilon$ is consistent, and so can be extended to a theory $\Phi'\in W_c$.
    If so, notice that clearly we have $(\Phi')^\Diamond\cap \Psi' = \varnothing$, and moreover $(\Phi')^\Box \subseteq \Psi'$ follows from the fact that $\Box\theta \in \Phi$ whenever $\Box\theta \in \Phi'$, in which case $\theta \in v \subseteq \Psi'$. Thus, $\Phi \peq_{\mathrm c} \Phi' \sqsubset_{\mathrm c} \Psi'$, as desired.
    
    We check that $\Upsilon$ is consistent. 
    Since every disjunction in $\Upsilon$ belongs to $\Phi$, $\Upsilon$ is prime. 
    Suppose towards a contradiction that $\Upsilon$ is inconsistent. 
    Then, there are finitely many $\chi_1, \hdots, \chi_n \in \Psi'$ such that $\Phi \cup \{\Diamond \chi_1, \hdots, \Diamond \chi_n\} \vdash\varnothing$. 
    Using this and iterative applications of $\ax{FS}$, we obtain:
    \begin{align*}
    \Phi &\vdash \Diamond \chi_1 \to \bigg(\Diamond\chi_2 \to \Big(\Diamond\chi_3 \to \big( \dots (\Diamond \chi_{n-1} \to (\Diamond \chi_n \to \bot)) \dots \big) \Big) \bigg)\\ 
    \Longrightarrow 
    \Phi &\vdash \Diamond \chi_1 \to \bigg(\Diamond\chi_2 \to \Big(\Diamond\chi_3 \to \big( \dots (\Diamond \chi_{n-1} \to (\Diamond \chi_n \to \Box\bot)) \dots \big) \Big) \bigg)\\ 
    \Longrightarrow 
    \Phi &\vdash \Diamond \chi_1 \to \bigg(\Diamond\chi_2 \to \Big(\Diamond\chi_3 \to \big( \dots (\Diamond \chi_{n-1} \to \Box(\chi_n \to \bot)) \dots \big) \Big) \bigg)\\ 
    \Longrightarrow 
    \Phi &\vdash \Diamond \chi_1 \to \bigg(\Diamond\chi_2 \to \Big(\Diamond\chi_3 \to \big( \dots \Box( \chi_{n-1} \to (\chi_n \to \bot)) \dots \big) \Big) \bigg)\\ 
    \Longrightarrow \dots \Longrightarrow
    \Phi &\vdash 
    \Box\bigg[ \chi_1 \to \bigg(\chi_2 \to \Big(\chi_3 \to \big( \dots ( \chi_{n-1} \to (\chi_n \to \bot)) \dots \big) \Big) \bigg)\bigg]. 
    \end{align*}
    Letting $(*)$ be the formula inside the square brackets, we have $\Phi\vdash \Box(*)$ and thus $\Psi \vdash (*)$. 
    As $\Psi \subseteq \Psi'$, we have $\Psi'\vdash (*)$. 
    However, $\Psi'$ contains each $\chi_i$, which contradicts the fact that $\Psi'$ is consistent.
\end{proof}

\begin{lemma}
    \label{lem::canonical-lambda-model-is-lambda}
    If $\Lambda$ is any $\ps$-regular logic, then $\mathcal M^\Lambda_{\mathrm c}$ is a $\Lambda$-model.
\end{lemma}
\begin{proof}
    It is immediate from the definitions that $\peq_c$ is a preorder and $V_c(p)$ is closed under $\peq$ for all $p\in\mathrm{Prop}$.
    As $4_\Box$ and $4_\Diamond$ are in $\Lambda$, $\sqsubset_c$ is a transitive relation.
    Furthermore, if $T_\Box$ and $T_\Diamond$ are in $\Lambda$, then it follows that $\sqsubset_c$ is reflexive.
    Forward confluence follows from an argument similar to the one in Lemma \ref{lem::canonical-model-confluences}.
    The other confluences follow from Lemma \ref{lem::canonical-model-confluences}.
\end{proof}

It already follows from the above that, for any of the $\ps$-regular logics $\Lambda$ we consider, $\mathcal M^\Lambda_{\mathrm c}$ is a $\Lambda$-model.
Thus to conclude our completeness proof, it remains to establish the following truth lemma:

\begin{lemma}
    \label{lem::truth-lemma-regular}
    Let $\Lambda$ be any $\ps$-regular logic such that $\Lambda$ contains either $\ax{FS}$ or $\ax{CD}$.
	For any  $\Phi \in W_{\mathrm c} $ and $\varphi \in \lanfull$, 
    \[
        \varphi \in \Phi  \iff  \mathcal M_{\mathrm c} ^{{\Lambda}}, \Phi \models \varphi.
    \]
\end{lemma}
\begin{proof}
    The proof is done by structural induction on the formulas.
    Note that we use Lemma \ref{lem:lindenbaum} on the cases for implications, boxes and diamonds.
    For example, if $\Diamond\varphi\in\Phi$, then we can show that $\Phi^\Box \cup \{\varphi\}$ is consistent, and so can be extended to a theory $\Psi$ such that $\Gamma\rel_c\Psi$.
\end{proof}

By Lemma~\ref{lem::truth-lemma-regular} along with the embedding of birelational models into polyderivative models and Lemma~\ref{lem::soundness}:
\begin{theorem}
    \label{thm::completeness-regular}
    If $\Lambda \in \{\ikf,\kfi,\gkf,\gkfc, \isf,\sfi,\gsf,\gsfc\}$,
    then $\Lambda$ is sound and strongly complete with respect to birelational $\Lambda$-models.
\end{theorem}

\subsection{Completeness of \texorpdfstring{$\ckf$}{CK4} and \texorpdfstring{$\csf$}{CS4}}\label{secCompCS4}

In this subsection we prove that $\ckf$ and $\csf$ are strongly complete for the class of birelational $\ckf$- and $\csf$-frames, respectively.
As far as we are aware, the strong completeness for $\ckf$ does not appear in the literature.

For this subsection, let $\Lambda\in\{\ckf,\csf\}$.

The proof is similar to the completeness for $\Diamond$-regular logics, but we need some modifications to treat the semantics of diamonds.
Here, we work with \define{augmented $\Lambda$-theories}, which are ordered pairs of sets of formulas  $\Phi = (\Phi^+ ; \Phi^{\ps})$ where $\Phi^+$ is a $\Lambda$-theory.
We say that $\Phi$ is \define{prime} if $\Phi^+$ is prime.
Again, we omit the reference to $\Lambda$ when not ambiguous.

The intuition behind augmented theories is the following: formulas in $\Phi^+$ are the ones satisfied by the theory, and the formulas in $\Phi^{\ps}$ are those $\varphi$ such that $\ps\varphi$ is falsified directly by $\rel$ not having any witnesses for $\varphi$---opposed to being falsified via a $ \peq  $-accessible world.

\begin{definition}
    \label{def:deduction2} 
	Given a set of formulas $\Xi$, we say that an augmented theory $\Phi$ is \define{$\Xi$-consistent} if for any finite set $\Delta \subseteq \Phi^{\ps}$,
	\[
		 \Phi^+ \not \vdash_\Lambda   \Xi,  \ps\bigvee \Delta.
	\]
    We adopt the convention that $\ps \bigvee \varnothing \coloneqq  \bot$.
    We say that $\Phi$ is \define{consistent} if it is $\varnothing$-consistent.
\end{definition}

Note that if $\Phi$ is a $\Xi$-consistent prime augmented $\Lambda$-theory, then necessarily $\bot \not\in\Phi^+$.
We say that an augmented $\Lambda$-theory $\Psi$ \define{extends} $\Phi$ if $\Phi^+\subseteq\Psi^+$ and $\Phi^\ps\subseteq\Psi^\ps$.

As with theories, we can also extend augmented theories to prime augmented theories:
\begin{lemma}
    \label{lem:lindenbaumGen} 
    Any $\Xi$-consistent augmented $\Lambda$-theory $(\Phi^+; \Phi^{\ps})$ can be extended to a $\Xi$-consistent prime augmented theory $(\Phi^+_*; \Phi^{\ps})$.
\end{lemma}

If $\Phi$ is an augmented $\Lambda$-theory then we define $\bfrm\Phi \coloneqq \{ \varphi \in \lanfull \mid \nec\varphi\in \Phi^+ \}$.

\begin{definition}
    We define the canonical model for $\Lambda$ as $\mathcal M_{\mathrm c}^\Lambda=(W_{\mathrm c}, {\peq_{\mathrm c} },{\rel_{\mathrm c} }, V_{\mathrm c} )$, where~
    \begin{itemize}
    	\item $W_{\mathrm c} $ is the set of all consistent prime augmented $\Lambda$-theories,
    	\item ${\peq_{\mathrm c} } \subseteq W_{\mathrm c} \times W_{\mathrm c} $ is defined by $\Phi \peq_{\mathrm c}  \Psi$ if and only if $\Phi^+ \subseteq \Psi^+$,
    	\item ${\rel_{\mathrm c} } \subseteq W_{\mathrm c} \times W_{\mathrm c} $ is defined by $\Phi \rel_{\mathrm c}  \Psi$ if and only if $\bfrm\Phi\subseteq \Psi^+$ and $\Phi^\ps \subseteq \Psi^{\ps}$,
    	\item $V_{\mathrm c} $ is defined by $V_{\mathrm c} (p)  = \lbrace \Phi \in W_{\mathrm c} \mid p \in \Phi^+  \rbrace$.
    \end{itemize}
\end{definition}

We can prove the following in Lemma \ref{lem::canonical-lambda-model-is-lambda}:
\begin{lemma}
    \label{lemCS4IsModel}
    If $\Lambda = \{\ckf,\csf\}$, then $\mathcal M_{\mathrm c}^\Lambda$ is a birelational $\Lambda$-model.
\end{lemma}

The following Truth Lemma is proved as \ref{lem:lindenbaumGen}
\begin{lemma}
    \label{lem::truth-lemma-irregular}
	For any prime augmented theory $\Phi \in W_{\mathrm c} $ and $\varphi \in \lanfull$, 
    \[
        \varphi \in \Phi^+ \iff \mathcal M_{\mathrm c} ^{{\mathsf{CS4}}}, \Phi \models \varphi.
    \]
\end{lemma}

By Lemma~\ref{lem::truth-lemma-irregular} along with the embedding of birelational models into polyderivative models and Lemma~\ref{lem::soundness}, we have:
\begin{theorem}
    \label{thm::completeness-irregular}
    If $\Lambda \in \{\ckf,\csf\}$,
    then $\Lambda$ is sound and strongly complete with respect to birelational $\Lambda$-models.
\end{theorem}
\end{document}